\documentclass{amsart}
\newtheorem{theorem}{Theorem}[section]
\newtheorem{lemma}[theorem]{Lemma}

\newtheorem{proposition}[theorem]{Proposition}
\newtheorem{sublemma}[theorem]{Sublemma}

\theoremstyle{definition}

\theoremstyle{remark}
\newtheorem{remark}[theorem]{Remark}
\numberwithin{equation}{section}

\newcommand{\eps}{\varepsilon}

\def\ve{\epsilon}

\usepackage{amsfonts}
\usepackage{color,graphicx,amssymb}
\usepackage{graphicx}
\begin{document}
\title{Metastability of Certain Intermittent Maps}
\author{Wael Bahsoun}
    \address{Department of Mathematical Sciences, Loughborough University, 
Loughborough, Leicestershire, LE11 3TU, UK}
\email{W.Bahsoun@lboro.ac.uk}
\author{Sandro Vaienti}
\address{
UMR-6207 Centre de Physique Th\'{e}orique, CNRS, Universit\'{e}
d'Aix-Marseille I, II, Universit\'{e} du Sud, Toulon-Var and FRUMAM,
F\'{e}d\'{e}ration de Recherche des Unit\'{e}s des Math\'{e}matiques de Marseille,
CPT Luminy, Case 907, F-13288 Marseille CEDEX 9}
\email{vaienti@cpt.univ-mrs.fr}
\thanks{We would like to thank H. Bruin for useful comments on this work. W.B. thanks the hospitality of Centre de Physique Th\'eorique, Luminy, where much of this work was carried. S. V. thanks the hospitality of the Department of Mathematical Sciences at Loughborough University where this work was initiated. S.V. was supported by the ANR-grant {\em Perturbations}}
\begin{abstract}
We study an intermittent map which has exactly two ergodic invariant densities. The densities are supported on two subintervals with a common boundary point. Due to certain perturbations, leakage of mass through subsets, called \textit{holes}, of the initially invariant subintervals occurs and forces the subsystems to merge into one system that has exactly one invariant density. We prove that the invariant density of the perturbed system converges in the $L^1$-norm to a \textit{particular} convex combination of the invariant densities of the intermittent map. In particular, we show that the ratio of the weights in the combination equals to the limit of the ratio of the measures of the \textit{holes}.
\end{abstract}
\maketitle
\pagestyle{myheadings} 
\markboth{Metastablilty of Certain Intermittent Maps }{W. Bahsoun And S. Vaienti}
\section{Introduction}
Open and metastable dynamical systems are currently very active topics of research in ergodic theory and dynamical systems. A dynamical system is called open if there is a subset in the phase space, called a \textit{hole}, such that whenever an orbit lands in it, the dynamics of this obit is terminated (see \cite{DY, Det} and references therein). A typical example of an open dynamical system is a billiard table with holes. Probabilistic and topological aspects of open dynamical systems have recently been of central interest to ergodic theorists  \cite{AB, BDM, BKT, DWY, FMS, FP, KL}.\\ 

A dynamical system is called metastable if it has two or more stable states. For example, a system which consists of two adjacent billiard tables that are linked via a small hole in their common boundary is a metastable dynamical system. Researchers have recognised that studying open dynamical systems can bring insights into the dynamics of metastable dynamical systems \cite{DW, GHW, KL}. In particular, it has been recognised that closed systems that are metastable behave approximately like a collection of open systems: the infrequent transitions between stable states in a metastable system are similar to infrequent escapes from associated open systems \cite{GHW, KL}.\\ 

 A particularly transparent description of this phenomenon is discussed in the recent work of Gonz\'alez-Tokman, Hunt and Wright \cite{GHW}.  In \cite{GHW}, a metastable expanding system is described by a piecewise smooth and expanding interval map which has two invariant sub-intervals and exactly two ergodic invariant densities. Due to small perturbations, the system starts to allow for infrequent leakage through subsets (also called \textit{holes}) of the initially invariant sub-intervals, forcing the two invariant sub-systems to merge into one perturbed system which has exactly one invariant density. The authors of \cite{GHW} proved that the unique invariant density of the perturbed interval map can be approximated by a convex combination of the two invariant densities of the original interval map, with the weights in the combination depending on the sizes of the holes.\\

In this paper, we depart to the \textit{non-uniformly} hyperbolic setting\footnote{With the exceptions of  \cite{BDM, FMS}, most of the results in ergodic theory of open and metastable systems have been obtained for uniformly hyperbolic systems. See also \cite{DY, FMS} for further details.}. In particular, we study an \textit{intermittent map} which has exactly two ergodic invariant densities. The densities are supported on two subintervals with a common boundary point. Due to certain perturbations, leakage of mass through \textit{holes} of the initially invariant subintervals occurs and forces the subsystems to merge into one system that has exactly one invariant density. We prove that the invariant density of the perturbed system converges in the $L^1$-norm to a \textit{particular} convex combination of the invariant densities of the intermittent map. In particular, we show that the ratio of the weights in the combination equals to the limit of the ratio of the measures of the \textit{holes}.\\

We would like to comment on the relationship between our work and the issue of statistical stability. The latter is usually established in the context of systems which admit a unique SRB measure (in our case an absolutely continuous invariant measure, {\em a.c.i.m.}) and  which are successively perturbed and the perturbed maps posses an SRB measure too. One way to formulate the statistical stability is by asking wether the perturbed density converges to the unperturbed one in $L^1$, w.r.t. the Lebesgue measure and whenever the SRB measure is absolutely continuous. A general result of this kind has been established by Alves and Viana   in the paper \cite{AV}, and successively by Alves \cite{A} where sufficient conditions are given to prove the statistical stability but still for the same class of maps. The latter is given by non-uniformly expanding maps which admit an induction structure with the first return map which is uniformly expanding, with bounded distortion and finally with {\em long branches} of the domains of local injectivity. The perturbed map is chosen in an open neighbourhood of the unperturbed one in the $C^k$ topology with $k\ge 2$, and a few more conditions are given to insure that the subsets with the same return times in the induction set are close and moreover the structural parameters of the maps (especially those bounding the derivative and the distortion) could be chosen uniformly in a $C^k$ neighbourhood of the unperturbed map. The main result is that when the perturbed maps converge to the unperturbed ones in the $C^k$ topology then the corresponding densities of the a.c.i.m. converge to each other in the $L^1$ norm, w.r.t. the Lebesgue measure.\\

 There are two main differences with our situation. First our unperturbed map admits more than one a.c.i.m.; second, the maps are only close in $C^0$, a better regularity being restored only locally on the open domain of injectivity of the branches. These two facts obliged us to find a completely different proof.\\

In section 2 we recall the result of \cite{GHW} about metastable expanding maps in a slightly more general setting. In section 3 we introduce our metastable intermittent system and its corresponding induced system. We then show that the induced system satisfies the assumptions of section 2. Moreover, we prove a lemma that relates invariant densities of the induced system to those of the original one. In section 4 we setup the problem of the metastable intermittent system. Further, we derive the formula of the \textit{particular} invariant density which is needed to approximate in the $L^1$-norm the invariant density of the perturbed system. This section also includes the statement of our main result (Theorem \ref{main}) and the strategy of our proof. Section 5 contains proofs of some technical lemmas and the proof of Theorem \ref{main}.\\

\noindent {\bf Notation.}  \\ $\Delta$ is an interval subset of $[0,1]$. We denote by $m$ the normalized Lebesgue measure on the unit interval and with $\left\|\cdot \right\|_1$ the associated $L^1$ norm. Given two sequences $a_n$ and $b_n$, when writing $a_n \lesssim b_n$, or equivalently $a_n=O(b_n)$ with $a_n$ and $b_n$ non-negative,  we mean that $\exists C\ge 1$, independent of $n$ and  such that $a_n\le C b_n$, $\forall n\ge 1$. By  $a_n \approx b_n$ we mean that $\exists C\ge 1$, independent of $n$ and such that $C^{-1}b_n\le a_n\le C b_n$, $\forall n\ge 1$. With $a_{n}\sim b_{n}$ we mean that $\lim_{n\rightarrow\infty}
\frac{a_{n}}{b_{n}}=1.$ We will also use the symbols ``$O$" 
in the usual Landau sense. Finally, $|Z|$ denotes the length of the interval $Z$.
 
\section{Invariant Densities of Metastable Expanding Maps}
\subsection{The expanding system}

Let $\hat T: \Delta\to \Delta$ be a map which satisfies the following conditions:\\

\noindent {\bf (A1)} There exists a countable partition of $\Delta$, which consists of a sequence of intervals $\{I_i\}_{i=1}^{\infty}$, $I_i\cap I_j=\emptyset$ for $ i\not= j$, $\bar I_i :=[q_{i,0},q_{i+1,0}]$ and there exists $\delta>0$  such that $\hat T_{i,0}:=\hat T|_{(q_{i,0},q_{i+1,0})}$ is $C^2$ which extends to a $C^2$ function $\bar T_{i,0}$ on a neighbourhood $[q_{i,0}-\delta,q_{i+1,0}+\delta]$ of $\bar I_i $ ;\\
\noindent {\bf (A2)} $\inf_{x\in \Delta\setminus\mathcal C_0}|\hat T'(x)|\ge\beta_0^{-1}>2$, where $\mathcal C_0=\{q_{i,0}\}_{i=1}^{\infty}$.\\
\noindent {\bf (A3)} The collection ${\hat T(I_i)}_{i=1}^{\infty}$ consists only of finitely many different intervals.\\
\noindent {\bf (A4)} $\exists$ $b$ in the interior of $\Delta$ such that $\hat T|_{\Delta_*}\subseteq \Delta_*$, where $*\in\{l,r\}$, $\Delta_*$ is an interval such that $\Delta_l\cup\Delta_r=\Delta$ and $\Delta_l\cap\Delta_r=\{b\}$.\\
\noindent {\bf (A5)} Let $H_0:=\hat T^{-1}\{b\}\setminus\{b\}$. We call $H_0$ the set of \textit{infinitesimal holes} and we assume that for every $n\ge 1$, $(\hat T^n\mathcal C_0)\cap H_0=\emptyset.$\\
\noindent {\bf (A6)} $\hat{T}$ verifies the Adler condition, namely there exists a constant $D_A>0$ such that  $\sup_i\sup_{x\in I_i}\frac{|D^2\hat{T}(x)|}{(D\hat{T}(x))^2}\le D_A$. In this case there will be an a.c.i.m. with a finite number of ergodic components \cite{Pi}. We will make the assumption that $\hat T$ admits exactly two ergodic a.c.i.ms $\hat \mu_*$, such that each $\hat \mu_*$ is supported on $\Delta_*$ and the corresponding density $\hat h_*$ is positive at each of the points of $H_0 \cap \Delta_*$. 
\subsection{Perturbations of the expanding system}  
Let $\hat T_{\eps}:\Delta\to \Delta$ be a perturbation of $\hat T$ which satisfies the following conditions:\\

\noindent {\bf (B1)} There exists a countable partition of $\Delta$, which consists of a sequence of intervals $\{I_{i,\eps}\}_{i=1}^{\infty}$, $I_{i,\eps}\cap I_{j,\eps}=\emptyset$ for $ i\not= j$, $\bar I_{i,\eps} :=[q_{i,\eps},q_{i+1,\eps}]$ such that\\ 
(i) for each $i$, $\eps\to q_{i,\eps}$ is a $C^2$ function for all $\eps\ge 0$ and for $\eps$ sufficiently small  we have that $[q_{i,\eps},q_{i+1,\eps}]\subset [q_{i,0}-\delta,q_{i+1,0}+\delta]$ ;\\ 
(ii) $\hat T_{\eps}|_{[q_{i,\eps},q_{i+1,\eps}]}$ has a $C^2$ extension $\bar T_{i,\eps}: [q_{i,0}-\delta,q_{i+1,0}+\delta]\to \mathbb R$, and $\bar T_{i,\eps}\to \bar T_{i,0}$ in the $C^2$ topology.\\
\noindent {\bf (B2)} The collection ${\hat T_{\eps}(I_{i,\eps})}_{i=1}^{\infty}$ consists only of finitely many different intervals.\\
\noindent{\bf (B3)} For each $\eps>0$, $\hat T_{\eps}$ admits a unique a.c.i.m. with  density $\hat h_{\eps}$.\\
\noindent {\bf (B4)}  Boundary condition:\\
(i) if $b\notin \mathcal C_0$, then $\hat T(b)=b$ and for all $\eps> 0$, $\hat T_{\eps}(b)=b$; \\
(ii) if $b\in \mathcal C_0$, then $\hat T(b-)<b<\hat T(b+)$ and for all $\varepsilon>0$, $b\in \mathcal C_{\varepsilon}$, where $\mathcal C_{\varepsilon}=\{q_{i,\eps}\}_{i=1}^{\infty}$.
\subsection{Holes in the expanding system $(\hat T_{\eps},\Delta)$}
We are interested in perturbations of $\hat T$ which produce ``leakage" of mass from $\Delta_l$ to $\Delta_r$ and vice versa. For this purpose we define the following sets:
$$\hat H_{l,\varepsilon}:=\Delta_l\cap \hat T^{-1}_{\varepsilon}(\Delta_r)$$
and 
$$\hat H_{r,\varepsilon}:=\Delta_r\cap \hat T^{-1}_{\varepsilon}(\Delta_l).$$
The sets $\hat H_{l,\varepsilon}$ and $\hat H_{r,\varepsilon}$ are called the ``left hole" and the ``right hole", respectively, of the perturbed expanding system $(\hat T_{\eps},\Delta)$ . Thus, when $\hat T_{\varepsilon}$ allows leakage of mass from $\Delta_l$ to $\Delta_r$, this leakage occurs when orbits of $\hat T_{\varepsilon}$ fall in the set $\hat H_{l,\varepsilon}$. Similarly, when $\hat T_{\varepsilon}$ allows leakage of mass from $\Delta_r$ to $\Delta_l$, this leakage occurs when orbits of $\hat T_{\varepsilon}$ fall in the set $\hat H_{r,\varepsilon}$.\\ 

Following \cite{GHW} the \textit{limiting hole ratio} ($l.h.r$) is defined by
$$l.h.r=\lim_{\eps\to 0}\frac{\hat \mu_r(\hat H_{r,\eps})}{\hat \mu_l(\hat H_{l,\eps})},$$
whenever the limit exists.\\ 

In the following we will denote by $BV([u,v])$ the space of functions of bounded variation defined on the closed interval $[u,v]$. We will equip this set with the complete norm given by the sum of the total variation plus the $L^1$ norm with respect to $m$. We denote this norm by $\left\|\cdot \right\|_{BV([u,v])}$ and the corresponding Banach space by $BV([u,v])$. By $P_{\eps}$ we denote the Perron-Frobenius operator \cite{Ba,BG} associated with the map $\hat T_{\eps}$ and acting on  $BV(\Delta)$. 

\begin{proposition}\label{prop1}
\text{ }\\
\begin{enumerate}
\item There exists a $\beta\in (0,1)$ and a $B\in (0,\infty)$, such that for any $\eps\ge 0$ and $f\in BV(\Delta)$, we have
$$\left\|P_{\eps}f\right\|_{BV(\Delta)}\le\beta\left\|f\right\|_{BV(\Delta)} +B||f||_1.$$
\item Suppose that the $l.h.r.$ exists. Then
$$\lim_{\eps\to 0}||\hat h_{\eps}-\hat h_p||_1=0,$$
where $\hat h_p=\hat\lambda_p\hat h_l+(1-\hat\lambda_p)\hat h_r$ and $\frac{\hat\lambda_p}{1-\hat\lambda_p}=l.h.r.$.
\end{enumerate}
\end{proposition}
\begin{proof}
The proof of the first statement, which is the uniform Lasota-Yorke inequality, is standard for $C^2$ perturbations of $\hat T$ with $|\hat T'(x)|\ge\beta^{-1}_0>2$ and satisfying Adler's condition. The proof of the second statement is exactly the same as the proof provided by \cite{GHW} for Lasota-Yorke maps with finite number of branches\footnote{We impose the same conditions as the ones imposed by \cite{GHW}, except that we relax the assumption on the number of branches. Instead of requiring the map to have only finite number of branches, we allow maps with countable number of branches whose image set is finite. The proofs of \cite{GHW} only depend on exploiting the locations and sizes of the jumps of the sets of discontinuities of the invariant densities $h_\eps$ which occur on the forward trajectories of the partition points of $\hat T_{\eps}$. Thus their proof follows verbatim for the class of maps $\hat T_{\eps}$ of this paper.}. 
\end{proof}
\begin{remark}
It will be important in the following that $\beta$ and $B$ can be chosen independently of $\eps$ and $\eps$ small. This can be easily achieved by recalling that 
those quantities are in fact explicitly determined in terms of the map, we refer to \cite{AV} for the details. In particular they depend on: (i) the infimum of the absolute value of the derivative, which we denoted by $\beta_0$ for $T$ and which persist larger than $2$ by condition (B1); (ii)  the constant $D_A$ bounding the Adler's condition which by its definition (see above),  can also be chosen uniformly in $\eps$ for $\eps$ small enough.
\end{remark}
\section{A metastable intermittent map}
A main issue of our work will be to compare a map of the interval with a neutral fixed point (intermittent map), with a perturbation of it. Instead of studying a general class of maps, we prefer to work with a particular example which allows us to analyze in a precise manner the steps of our approach. By looking at the proofs in the following sections, it will be clear that our approach can be extended to other intermittent maps. 
\subsection{The intermittent map and its perturbation}\label{intermittent}
Let $\alpha\in (0,1)$. For each $\eps\ge 0$ define the continuous map $T_{\eps}:[0,1]\to[0,1]$  by:\\

\begin{equation}\label{map}
T_{\eps}(x)=\left\{\begin{array}{cccc}
T_{1,\eps}:=x+4^{\alpha}(1+4\eps)x^{1+\alpha}&\mbox{for $0\le x< \frac{1}{4}$}\\
T_{2,\eps}:=-4(1+2\eps)x+\frac32 +3\eps&\mbox{for $\frac{1}{4}\le x< \frac{3}{8}$}\\
T_{3,\eps}:=4x-\frac32&\mbox{for $\frac{3}{8}\le x< \frac{1}{2}$}\\
T_{4,\eps}(x)&\mbox{for $\frac{1}{2}\le x\le 1$}
\end{array}
\right. .
\end{equation}
\begin{figure}[htbp] 
   \centering
   \includegraphics[width=2.5in]{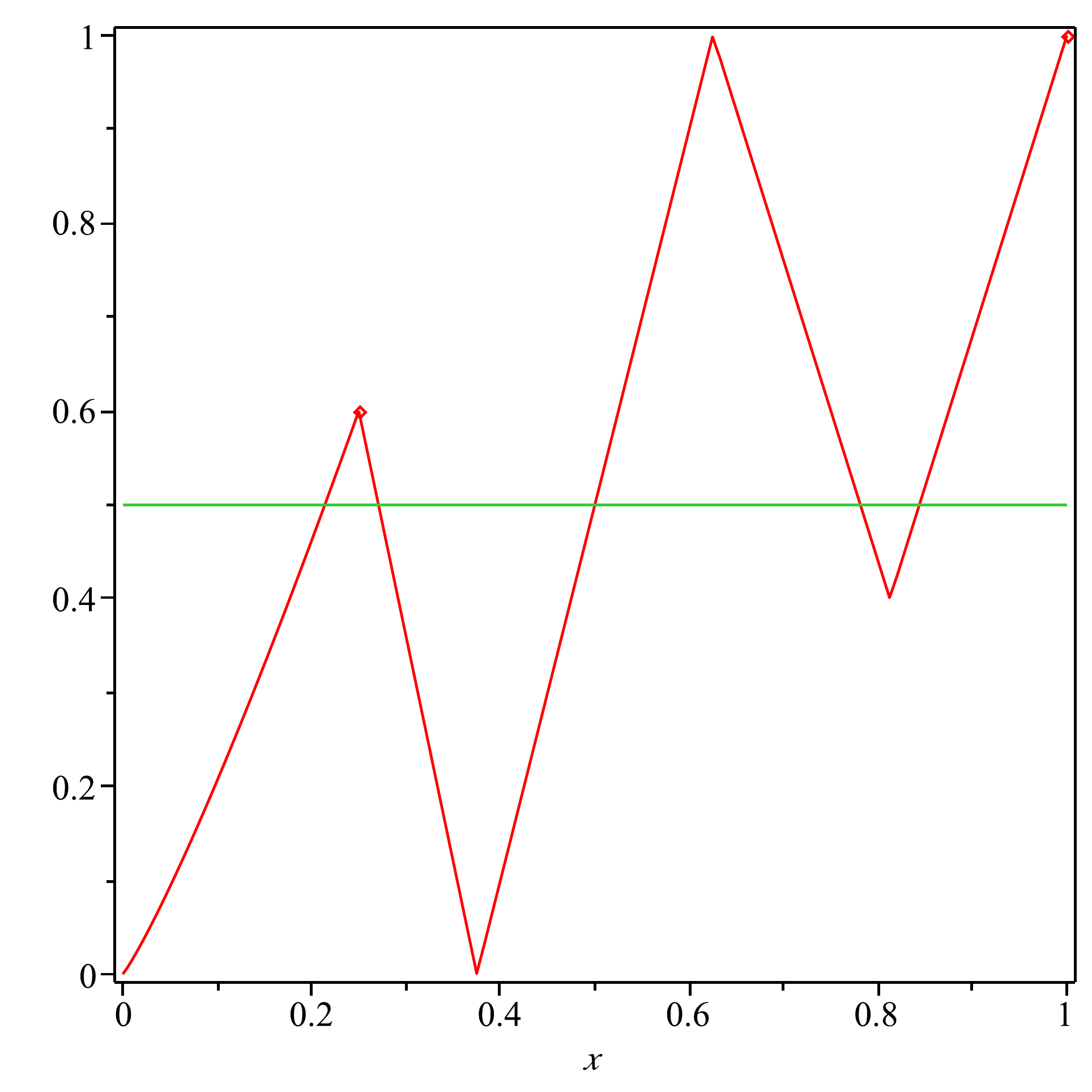} 
   \caption{The graph of $T_{\eps}$ for the values $\alpha=0.3$ and $\eps=0.1$}   \label{Fig1}
\end{figure}
The component $T_{4,\eps}(x)$ continuously extends $T_{3,\eps}$ on the right; it is piecewise expanding with the absolute value of the derivative bigger than\footnote{Since $T_{4,\eps}\equiv \hat {T_{\eps}}|_{[1/2,1]}$, one can replace the assumption $\inf_{x}|T'_{4,\eps}(x)|>2$ by the assuming that $\inf_{x}|T'_{4,\eps}(x)|>1$ and $T_{4,\eps}$ has no periodic critical points except at $1$. See \cite{GHW} for further details.} $2$, of class $C^2$ except for the points of relative minima and with a finite number of long branches. We will assume that it has only one spike emerging on the right side of $1/2$ (see Figure \ref{Fig1}) and this spike is located at the point of relative minimum $s_r$ which does not move with $\eps$. We finally suppose that the height of the spike is exactly $\eps$; likewise for the left side. Notice that for $\eps=0$, the intermittent map $T_0:=T$ has exactly two ergodic invariant probability\footnote{Note that the case $\alpha\ge 1$ in (\ref{map}) is not covered in this paper. It is well known that when $\hat \mu_l$ is $\sigma$-finite. Obtaining results similar to those of this paper for intermittent maps with $\alpha\ge 1$ is an interesting open problem.} densities, $h_l$ supported on $[0,1/2]$ and $h_r$ supported on $[1/2,1]$. Moreover, for any $\eps>0$, the perturbed map has a unique invariant probability density $h_{\eps}$. We will elaborate more on the uniqueness of $h_\eps$ in the Appendix. The graph of the map is shown in Figure \ref{Fig1}. Let us point out that with our assumptions $T$ and $T_{\eps}$ are $C^0$ close, namely $\lim_{\eps\rightarrow 0}\left\| T_{\eps}-T\right\|_0 =0$. Since $T$ and $T_{\eps}$ are also continuous (and hence uniformly continuos on the closed unit interval), this implies that for any $n>0$ we have as well $\lim_{\eps\rightarrow 0}\left\| T^n_{\eps}-T^n\right\|_0 =0$.\\

\subsection{Holes in the intermittent system $(T_{\eps}, [0,1])$}
We are interested in perturbations of $T$ which produce ``leakage" of mass from $I_l:=[0,b]$ to $I_r:[b,1]$ and vice versa. For this purpose we define the following sets:
$$H_{l,\varepsilon}:=I_l\cap T^{-1}_{\varepsilon}(I_r)$$
and 
$$ H_{r,\varepsilon}:=I_r\cap T^{-1}_{\varepsilon}(I_l).$$
The sets $H_{l,\varepsilon}$ and $H_{r,\varepsilon}$ are called the ``left hole" and the ``right hole", respectively, of the perturbed intermittent system $(T_{\eps},[0,1])$. Note that for the intermittent system defined in (\ref{map}) $b:=1/2$.
\subsection{The induced system}
For each $\eps\ge 0$, we induce $T_{\eps}$ on the same set $\Delta:=[a_0,1]$, where $a_0:=1/4$. We also set $b_0:=1/4$. It is important to notice that $a_0$ and consequently $\Delta$ are independent of $\eps$ (See Figure \ref{Fig2}). Then for $n\ge 1$ we define
$$b_{n+1,\eps} = T_{1,\eps}^{-1}(b_{n,\eps}),\, a_{n,\eps} = T^{-1}_{2,\eps}(b_{n,\eps}),\text{ and } a'_{n,\eps} = T^{-1}_{3,\eps}(b_{n,\eps}).$$
Then for $\eps\ge 0$ we define the induced map $\hat T_{\eps}:\Delta\to\Delta$ by
\begin{equation}\label{induced}
\hat T_{\eps}(x)=\left\{\begin{array}{cc}
T_{\eps}(x)&\mbox{for $x\in Z_{1,\eps}$}\\
T^{n+1}_{\eps}(x)&\mbox{for $x\in Z_{n,\eps}$}\\
\end{array}
\right. ,
\end{equation}
where $Z_{1,\eps}:=(a_0,a_{1,\eps})\cup(a'_{1,\eps},1)$ and $Z_{n,\eps}:=(a_{n-1,\eps},a_{n,\eps})\cup(a'_{n,\eps},a'_{n-1,\eps})$.\\

\noindent We now define the following sets:
$$W_{0,\eps}:=(a_0,1)\text{ and } W_{n,\eps}:=(b_{n,\eps},b_{n-1,\eps}),\, n\ge 1.$$
Observe that 
$$T_{\eps}(Z_{n,\eps})= W_{n-1,\eps}\text{ and } \tau_{Z_{n,\eps}}=n,$$
where $\tau_{Z_{n,\eps}}$ is the first return time of $Z_{n,\eps}$ to $\Delta$.
\begin{lemma}\label{Le0}
\text{ }
\begin{enumerate}
\item For $\eps=0$, the invariant densities of $\hat T$, $\hat h_l$ and $\hat h_r$, are Lipschitz continuous and bounded away from $0$ on $[a_0,b]$, $[b,1]$ respectively.
\item For $\eps=0$, the induced map $\hat T:\Delta\to\Delta$ satisfies assumptions (A1)-(A6).
\item For $\eps>0$, the perturbed induced map $\hat T_{\eps}$ satisfies conditions (B1)-(B4).
\item The limiting hole ratio of the induced system
$$l.h.r=\lim_{\eps\to 0}\frac{\hat \mu_r(\hat H_{r,\eps})}{\hat \mu_l(\hat H_{l,\eps})}$$
exists and it is different from zero and infinity.
\end{enumerate} 
\end{lemma}
\begin{proof}
Statement (1) follows from the fact that $\hat T_{|[a_0,b]}$ is piecewise $C^2$, piecewise onto and expanding (see \cite{BG} for example). The same properties hold for $\hat T_{|[b,1]}$.  
To prove (2), observe that $\sup_{x\in\Delta}|\hat T'(x)|>3$. Moreover, for all $n\ge 1$, $\hat T^{n}(\mathcal C_0)=\{b,1\}\cap H_0=\emptyset$. Statement (3) is satisfied, in particular, condition (B4).
We now prove (4). We first observe that 
$$
\frac{\hat\mu_r(\hat H_{r,\eps})}{\hat\mu_l(\hat H_{l,\eps})}=\frac{\int_{\hat H_{r,\eps}}\hat h_rdx}{\int_{\hat H_{l,\eps}}\hat h_ldx}=\frac{\hat h_r(\xi_{r,\eps})|\hat H_{r,\eps}|}{\sum_{k=1}^{\infty}\hat h_l(\xi_{l,\eps}^k)|Q_{k,\eps}|},
$$
where we applied the mean value theorem: $\xi_{r,\eps}$ is a point in $\hat H_{r,\eps}$, $Q_{k,\eps}=[a_{k-1,\eps},w_{k,\eps}]$, where $w_{k,\eps}=\hat{T}_{\eps}^{-1}(b)\cap Z_{k,\eps}$ and $\xi_{l,\eps}^k$ is a point in $Q_{k,\eps}$. Again by the mean value theorem there will be a point $\chi_{l,\eps}^k\in Q_{k,\eps}$ and such that $|Q_{k,\eps}|=\frac{\eps}{|D\hat{T}_{\eps}(\chi_{l,\eps}^k)|}$. Moreover, by the assumptions on the branch $\hat{T}_{4,\eps}$ we get immediately that $|\hat H_{r,\eps}|=\eps \left[|D\hat{T}_{4,\eps}(u_{l,\eps})|^{-1}+|D\hat{T}_{4,\eps}(u_{r,\eps})|^{-1}\right]$, where $u_{l,\eps}$ (resp. $u_{r,\eps}$) is a point on the left hand side (resp. right hand side) of $s_r$. Recall that $s_r$ is the relative minimum of $T_{4,\eps}$ and that $T_{4,\eps}\equiv \hat T_{4,\eps}$. Thus we have
\begin{equation}\label{neweq}
\frac{\hat\mu_r(\hat H_{r,\eps})}{\hat\mu_l(\hat H_{l,\eps})}=\frac{\hat h_r(\xi_{r,\eps})| \left[|D\hat{T}_{4,\eps}(u_{l,\eps})|^{-1}+|D\hat{T}_{4,\eps}(u_{r,\eps})|^{-1}\right]}{\sum_{k=1}^{\infty}\hat h_l(\xi_{l,\eps}^k)|D\hat{T}_{\eps}(\chi_{l,\eps}^k)|^{-1}}.
\end{equation}
We first deal with the denominator on the right hand side of (\ref{neweq}). We write 
$$ D\hat{T}_{\eps}(\chi_{l,\eps}^k)-D\hat{T}(a_{k-1})=D\hat{T}_{\eps}(\chi_{l,\eps}^k)-D\hat{T}(\chi_{l,\eps}^k)+D\hat{T}(\chi_{l,\eps}^k)-D\hat{T}(a_{k-1}).$$ 
Note that, by assumption (B1), 
$$\lim_{\eps\to 0}|D\hat{T}_{\eps}(\chi_{l,\eps}^k)-D\hat{T}(\chi_{l,\eps}^k)|=0,$$
and, by the continuity of $D\hat{T}$ on  $[a_{k-1}-\delta, a_k+\delta]$,
$$\lim_{\eps\to 0}|D\hat{T}(\chi_{l,\eps}^k)-D\hat{T}(a_{k-1})|=0.$$
Therefore, 
$$\lim_{\eps\to 0}\sum_{k=1}^{\infty}\hat h_l(\xi_{l,\eps}^k)|D\hat{T}_{\eps}(\chi_{l,\eps}^k)|^{-1}=\sum_{k=1}^{\infty}\hat h_l(a_{k-1})|D\hat{T}(a_{k-1})|^{-1}.$$ 
We now show that $\sum_{k=1}^{\infty}\hat h_l(a_{k-1})|D\hat{T}(a_{k-1})|^{-1}$ is finite and different from $0$. First of all the density $\hat h_l$ is bounded away from zero and infinity in the preimages of $b$ since it is Lipschitz continuous and bounded from below on $[b_0,b]$. Then we observe that the assumptions (A1, A2, A3, A6) imply that the first return map has bounded distortion. Therefore, there exists a constant $C_d$ independent of $k$ which allows us to bound $|D\hat{T}(a_{k-1})|^{-1}\le C_d |D\hat{T}(v_{k})|^{-1}$ where $v_k$ is a point in $Z_{k,\eps}$ for which the inverse of the derivative  gives the length $|Z_{k,\eps}|$ of $Z_{k,\eps}$ times the inverse of the length  of $[b_0,b]$; finally  the sum over the lengths of the $Z_{k,\eps}$ on $[b_0,b]$ gives of course $b-b_0$. 
We now bound the numerator in (\ref{neweq}). By an argument similar to that used above we have
\begin{equation*}
\begin{split}
&\lim_{\eps\to 0}\hat h_r(\xi_{r,\eps})\left[|D\hat{T}_{4,\eps}(u_{l,\eps})|^{-1}+|D\hat{T}_{4,\eps}(u_{r,\eps})|^{-1}\right]\\
&\hskip 6cm =\hat h_r(s_r)\left[|D_l\hat{T}_{4}(s_r)|^{-1}+|D_r\hat{T}_{4}(s_r)|^{-1}\right],
\end{split}
\end{equation*}
where $D_l\hat{T}_4(s_r)$ (resp. $D_r\hat{T}_4(s_r)$) denotes the right (resp. left) derivative of $\hat{T}_4$ at the point $s_r$.
\end{proof}
\begin{remark}\label{Re1}
Lemma \ref{Le0} implies that results of Proposition \ref{prop1} hold for the induced system. In particular,
 $$\lim_{\eps\to 0}||\hat h_{\eps}-\hat h_p||_1=0,$$
where 
\begin{equation}\label{particularhat}
\hat h_p:=\hat\lambda_p\hat h_l+(1-\hat\lambda_p)\hat h_r.
\end{equation}
and 
$$l.h.r.=\frac{\hat\lambda_p}{1-\hat\lambda_p}.$$
\end{remark}
\begin{figure}[htbp] 
   \centering
   \includegraphics[width=3in]{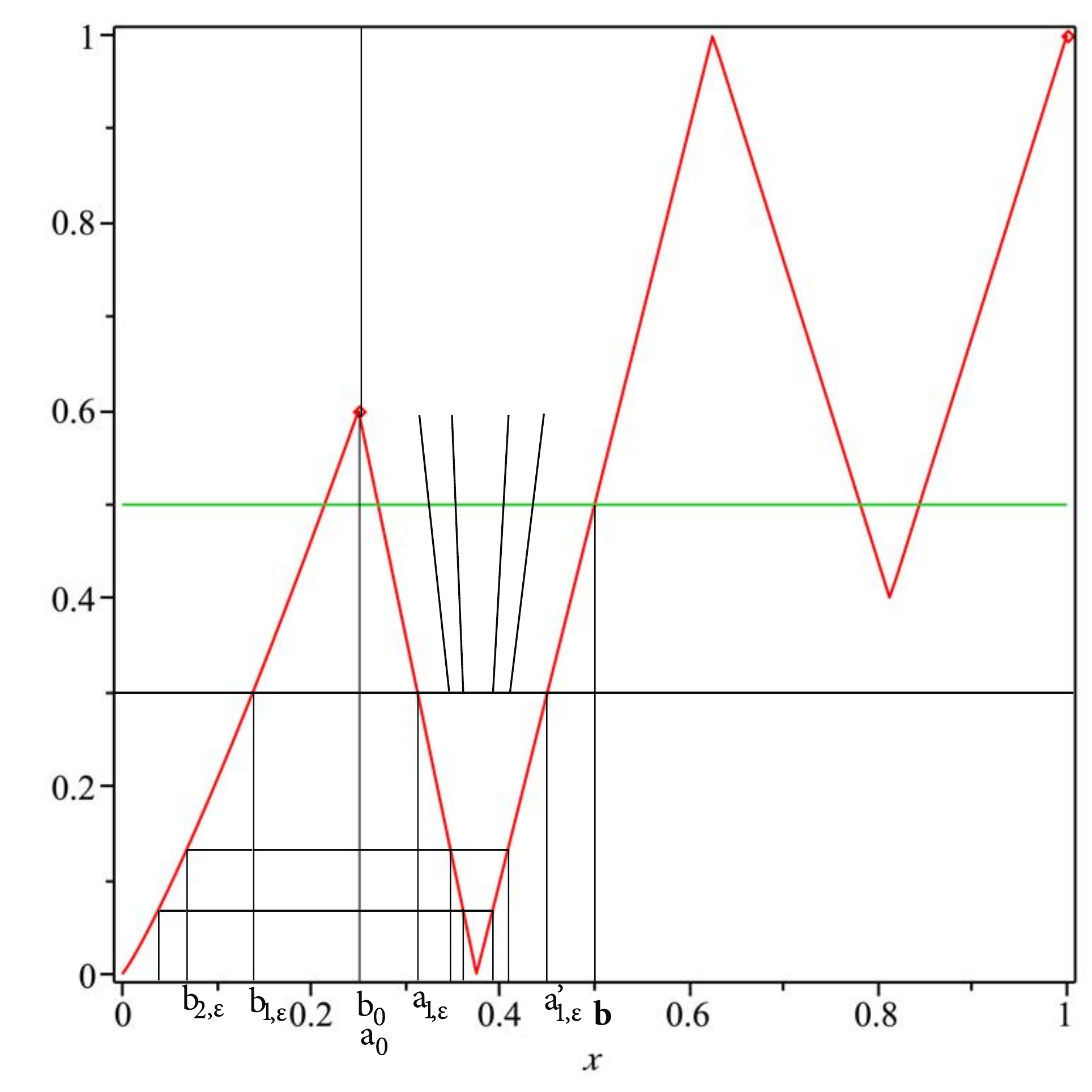} 
   \caption{The graph of induced system $\hat T_{\eps}$ for the values $\alpha=0.3$ and $\eps=0.1$}   \label{Fig2}
\end{figure}
\subsection{Pulling back the invariant density}
For all $\eps\ge 0$, we can find an a.c.i.m., $\mu_{\eps}$, of $T_{\eps}$  using the a.c.i.m., $\hat \mu_{\eps}$, of $\hat T_{\eps}$ \cite{Pi}. In particular, for any measurable set $B\subset [0,1]$, we have
\begin{equation}\label{measure}
\mu_{\eps}(B)=c_{\tau,\eps}\sum_{n=1}^{\infty}\sum_{j=0}^{\tau_{Z_{n,\eps}-1}}\hat\mu_{\eps}(T_{\eps}^{-j}B\cap Z_{n,\eps}),
\end{equation}
where $c^{-1}_{\tau,\eps}=\sum_{k=1}^{\infty} \tau_{Z_{k,\eps}} \hat\mu_{\eps}(Z_{k,\eps})$. In the following lemma we provide a lemma expressing the density of $\mu_{\eps}$ in terms of that of $\hat\mu_{\eps}$. This will play a crucial role in the proof of our main result.
\begin{lemma}\label{Le1}
Let $\mu_\eps$ be a $T_{\eps}$-acim, defined as in (\ref{measure}). Then, for $\eps\ge 0$, 
\begin{equation}\label{density}
h_{\eps}(x)=\left\{\begin{array}{ccc}
c_{\tau,\eps}\hat h_{\eps}(x)&\mbox{for $x\in \Delta$}\\
\text{ }\\
c_{\tau,\eps}\sum_{n=k+1}^{\infty}\left(\sum_{i=2}^{3} \frac{\hat h_{\eps}(T_{i,\eps}^{-1}T_{1,\eps}^{-(n-k-1)}x)}{|DT_{\eps}^{(n-k)}(T_{i,\eps}^{-1}T_{1,\eps}^{-(n-k-1)}x)|} \right) &\mbox{for $x\in W_{k,\eps}$}
\end{array}
\right. ,
\end{equation} 
where $h_{\eps}$ and $\hat h_{\eps}$ are the densities of $\mu_{\eps}$ and $\hat\mu_{\eps}$ respectively. 
\end{lemma}
\begin{proof}
By (\ref{measure}), for any measurable set $B\subset\Delta$, we have
$$\mu_{\eps}(B)=c_{\tau,\eps}\hat\mu_{\eps}(B).$$
Passing to the densities and for Lebesgue almost all 
$x\in \Delta$, we obtain
$$h_{\eps}(x)=c_{\tau,\eps}\hat h_{\eps}(x).$$  
We then extend $h_{\eps}$ to a bounded variation function as $\hat{h}_{\eps}$. This proves formula (\ref{density}) for $x\in \Delta$.\\ 

We now consider the case when $B\subseteq W_{k,\eps}$. First, suppose $B= W_{k,\eps}$, for some $k$. Then by (\ref{measure}), we have
$$\mu_{\eps}(W_{k,\eps})=c_{\tau,\eps}\sum_{n=1}^{\infty}\sum_{j=0}^{n-1}\hat\mu_{\eps}(T_{\eps}^{-j}W_{k,\eps}\cap Z_{n,\eps})=c_{\tau,\eps}\sum_{n=k+1}^{\infty}\hat\mu_{\eps}(Z_{n,\eps}).$$
Therefore, if $B\subseteq W_{k,\eps}$, we obtain
$$\mu_{\eps}(B)=c_{\tau,\eps}\sum_{n=k+1}^{\infty}\hat\mu_{\eps}(T_{\eps}^{-(n-k)}B\cap Z_{n,\eps}).$$
consequently,
$$
\int_B h_{\eps}dx=\sum_{n=k+1}^{\infty}\int_{T_{\eps}^{-(n-k)}B\cap Z_{n,\eps}}\hat h_{\eps(x)}dx.
$$
We now perform the change of variable $T^{n-k}y=x$ by observing that the set $B$ is pushed backward $n-k-1$ times with $T_{1,\eps}^{-1}$ and  then it splits
into three parts according to the actions of $T_{1,\eps}^{-1}, T_{2,\eps}^{-1},  T_{3,\eps}^{-1}$.
Therefore,
\begin{equation*}
\int_B h_{\eps}dx=c_{\tau,\eps}\sum_{n=k+1}^{\infty}\int_{B}\left(\sum_{i=2}^{3} \frac{\hat h_{\eps}(T_{i,\eps}^{-1}T_{1,\eps}^{-(n-k-1)}y)}{|DT_{\eps}^{(n-k)}(T_{i,\eps}^{-1}T_{1,\eps}^{-(n-k-1)}y)|} \right)dy,
\end{equation*}
where $DT_{\eps}^{(n-k)}(z)$ is the derivative of $T_{\eps}^{(n-k)}$ evaluated at the point $z$. Thus, for Lebesgue almost all $x\in W_{k,\eps}$ we obtain
\begin{equation*}
\begin{split}
h_{\eps}(x)&= c_{\tau,\eps}\sum_{n=k+1}^{\infty}\left(\sum_{i=2}^{3} \frac{\hat h_{\eps}(T_{i,\eps}^{-1}T_{1,\eps}^{-(n-k-1)}x)}{|DT_{\eps}^{(n-k)}(T_{i,\eps}^{-1}T_{1,\eps}^{-(n-k-1)}x)|} \right)\\
&=c_{\tau,\eps}\sum_{n=1}^{\infty}\left(\sum_{i=2}^{3} \frac{\hat h_{\eps}(T_{i,\eps}^{-1}T_{1,\eps}^{-(n-1)}x)}{|DT_{\eps}^{(n)}(T_{i,\eps}^{-1}T_{1,\eps}^{-(n-1)}x)|} \right).
\end{split}
\end{equation*}
The last expression shows that $h_{\eps}$ can be extended to a bounded variation function over all $\Delta^c$ and therefore over all the unit interval.
\end{proof}
\section{The problem of the original intermittent system}\label{problem} 
\subsection{The problem} In subsection \ref{intermittent} we noted that the intermittent map $T$ has exactly two ergodic invariant densities, $h_l$ supported on $[0,1/2]$ and $h_r$ supported on $[1/2,1]$. Moreover, for any $\eps>0$, the perturbed map has a unique invariant density $h_{\eps}$. The uniqueness of the invariant density $h_{\eps}$ is proved in the Appendix.\\

Our main goal is to prove that the invariant density of the perturbed system $h_{\eps}$ converges in the $L^1$-norm to a \textit{particular} convex combination of the invariant densities, $h_l$ and $h_r$, of the intermittent map. We define
\begin{equation}\label{extension}
h_{p}(x):=\left\{\begin{array}{ccc}
c_{\tau,p}\hat h_{p}(x)&\mbox{for $x\in \Delta$}\\
\text{ }\\
c_{\tau,p}\sum_{n=k+1}^{\infty}\left(\sum_{i=2}^{3} \frac{\hat h_p(T_{i}^{-1}T_{1}^{-(n-k-1)}x)}{|DT^{(n-k)}(T_{i}^{-1}T_{1}^{-(n-k-1)}x)|} \right) &\mbox{for $x\in W_{k}$}
\end{array}
\right. ,
\end{equation} 
where $c_{\tau,p}^{-1}=\sum_{k=1}^{\infty}k\hat\mu_p(Z_k)$, $\hat\mu_p=\hat\lambda_p\hat\mu_l+(1-\hat\lambda_p)\hat\mu_r$.
\begin{remark}\label{Re2}
Note that, by Lemma \ref{Le1}, $h_p$ is a $T$-invariant density. Moreover, since $T$ has exactly two ergodic invariant densities $h_l$ and $h_r$, $h_p$ is a convex combination of $h_l$ and $h_r$. In fact, $h_p$ is a \textit{particular} convex combination of $h_l$ and $h_r$. In the following proposition, we give an explicit representation of $h_p$ in terms of $h_l$ and $h_r$. 
\end{remark}
\begin{proposition}\label{prop2} 
The representation of $h_p$ in terms of $h_l$ and $h_r$ is given by
\begin{equation*}
h_p(x)=\lambda_p h_l(x)+(1-\lambda_p)h_r(x),
\end{equation*}
where $\lambda_p=\frac{\hat\lambda_pc_{\tau,r}}{\hat\lambda_pc_{\tau,r}+(1-\hat\lambda_p)c_{\tau,l}}$, $c_{\tau,l}^{-1}=\sum_{k=1}^{\infty}k\hat\mu_l(Z_k)$ and $c_{\tau,r}^{-1}=\sum_{k=1}^{\infty}k\hat\mu_r(Z_k)$.
\end{proposition}
\begin{proof}
First, using Lemma \ref{Le1}, we have
\begin{equation*}
h_{l}(x)=\left\{\begin{array}{ccc}
c_{\tau,l}\hat h_{l}(x)&\mbox{for $x\in \Delta$}\\
\text{ }\\
c_{\tau,l}\sum_{n=k+1}^{\infty}\left(\sum_{i=2}^{3} \frac{\hat h_{l}(T_{i}^{-1}T_{1}^{-(n-k-1)}x)}{|DT^{(n-k)}(T_{i}^{-1}T_{1}^{-(n-k-1)}x)|} \right) &\mbox{for $x\in W_{k}$}
\end{array}
\right. ,
\end{equation*} 
and for all $x\in[0,1]$
$$h_r(x)=c_{\tau,r}\hat h_r(x).$$
Moreover, 
$$c_{\tau,p}^{-1}=\sum_{k=1}^{\infty}k\hat\mu_p(Z_k)=\hat\lambda_p\sum_{k=1}^{\infty}k\hat\mu_l(Z_k)+(1-\hat\lambda_p)\sum_{k=1}^{\infty}k\hat\mu_r(Z_k)=\hat\lambda_pc_{\tau,l}^{-1}+(1-\hat\lambda_p)c_{\tau,r}^{-1}.$$
Therefore, using (\ref{extension}), for $x\in\Delta$, we have
\begin{equation*}
\begin{split}
h_p(x)&=c_{\tau,p}(\hat\lambda_p\hat h_l(x)+(1-\hat\lambda_p )h_r(x))\\
&=\hat\lambda_p\frac{c_{\tau,l}}{\hat\lambda_p+(1-\hat\lambda_p)c_{\tau,l}c_{\tau,r}^{-1}}\hat h_l(x)+(1-\hat\lambda_p)\frac{c_{\tau,r}}{\hat\lambda_pc_{\tau,r}c_{\tau,l}^{-1}+(1-\hat\lambda_p)}\hat h_r(x)\\
&=\frac{\hat\lambda_p}{\hat\lambda_p+(1-\hat\lambda_p)c_{\tau,l}c_{\tau,r}^{-1}}h_l(x)+\frac{(1-\hat\lambda_p)}{\hat\lambda_pc_{\tau,r}c_{\tau,l}^{-1}+(1-\hat\lambda_p)}\hat h_r(x)\\
&=\lambda_ph_l(x)+(1-\lambda_p)h_r(x).
\end{split}
\end{equation*}
Using (\ref{extension}) again, for $x\in W_k$, we obtain
\begin{equation*}
\begin{split}
h_p(x)&=c_{\tau,p}\hat\lambda_p\sum_{n=k+1}^{\infty}\left(\sum_{i=2}^{3} \frac{\hat h_l(T_{i}^{-1}T_{1}^{-(n-k-1)}x)}{|DT^{(n-k)}(T_{i}^{-1}T_{1}^{-(n-k-1)}x)|} \right)\\
&=\hat\lambda_p\frac{c_{\tau,l}}{\hat\lambda_p+(1-\hat\lambda_p)c_{\tau,l}c_{\tau,r}^{-1}}\sum_{n=k+1}^{\infty}\left(\sum_{i=2}^{3} \frac{\hat h_l(T_{i}^{-1}T_{1}^{-(n-k-1)}x)}{|DT^{(n-k)}(T_{i}^{-1}T_{1}^{-(n-k-1)}x)|} \right)\\
&=\frac{\hat\lambda_p}{\hat\lambda_p+(1-\hat\lambda_p)c_{\tau,l}c_{\tau,r}^{-1}}h_l=\lambda_p h_l(x).
\end{split}
\end{equation*}
\end{proof}
\subsection{Main result and the strategy of our proof}
The following theorem is the main result of the paper. 
\begin{theorem}\label{main}
Let $h_{\eps}$ be the unique invariant density of $T_{\eps}$. Then
\begin{enumerate}
\item 
$$\lim_{\eps\to 0}||h_{\eps}-h_p||_1=0.$$
\item Moreover,
$$\lim_{\eps\to 0}\frac{\mu_r(H_{r,\eps})}{\mu_l(H_{l,\eps})}=\frac{\lambda_p}{1-\lambda_p}.$$
\end{enumerate}
\end{theorem}
To prove (1) of Theorem \ref{main}, we use the following strategy:
\begin{enumerate}
\item First we estimate
\begin{equation}\label{strategy}
\begin{split}
||h_{\eps}-h_p||_1&\le\int_{\Delta}|h_{\eps}-h_p|dx+\sum_{k=1}^{\infty}\int_{W_{k}\setminus (W_{k,\eps}\cap W_k)}|h_{\eps}-h_p|dx\\
&+\sum_{k=1}^{\infty}\int_{W_{k,\eps}\cap W_k}|h_{\eps}-h_p|dx= {(I)}+ (II) +(III).
\end{split}
\end{equation}
\item In (I), we exploit the representations of $h_p$, $h_{\eps}$ on $\Delta$, and use Remark \ref{Re1} to conclude that the limit of (I) is zero as $\eps\to 0$.
\item In (II), we obtain an upper bound
$$\sup_{x\in W_{k}\setminus (W_{k,\eps}\cap W_k)}|h_{\eps}(x)|+\sup_{x\in W_{k}\setminus (W_{k,\eps}\cap W_k)}|h_{\eps}(x)|\lesssim k.$$
 Since the left boundary point of $W_k$, $b_k$, scales like $k^{-\frac{1}{\alpha}}$, we have just recovered, with a different technique, the well known fact that the density of the intermittent map behaves like $x^{-\alpha}$ in the neighbourhood of the neutral fixed point. Consequently, this implies that 
$$(II)\lesssim \sum_{k=1}^{\infty}k|W_{k}\setminus (W_{k,\eps}\cap W_k)|\simeq\sum_{k=1}\frac{1}{k^{1/\alpha}}.$$
and the uniform convergence of the series allows us to bring the limit inside for $\eps\rightarrow 0$.
\item In (III) $h_{\eps}$ and $h_{p}$ can be compared on $W_{k,\eps}\cap W_k$ via their representations in terms of $\hat h_{\eps}$ and $\hat h_p$ respectively. We then show that (III) is summable. This allows us to move the limit $\eps\to 0$ inside the sum to conclude that the limit of (III) equals zero. In this part, we invoke two results from the induced system. Namely that $\lim_{\eps\to 0}||\hat h_{\eps}-\hat h_p||_1=0$, and the fact that $\hat h_{p}$ is Lipschitz continuous on $[a_0,b]$. 
\end{enumerate}

\noindent To prove (2) of Theorem \ref{main}, we use the representation of $\lambda_p$ in Proposition \ref{prop2} and part (1) of Theorem \ref{main}.
\section{Proof of Theorem \ref{main}}
Before proving Theorem \ref{main}, we state and prove two lemmas. We first observe that $T_{\eps}(a_{k,\eps})=b_{k,\eps}$ and $b_{k,\eps}\lesssim k^{-\frac{1}{\alpha}}$, see for instance Lemma 3.2 in \cite{LSV}. Thus, $|Z_{k,\eps}|\lesssim k^{-\frac{1}{\alpha}-1}$. In fact we precisely have $|Z_{k,\eps}|\le C_{\ve} k^{-\frac{1}{\alpha}-1}$, where $C_{\ve}=1+O(\eps)$. In the next Lemma, $\tilde{C}$ will denote a constant which is independent of $\eps$. $\tilde C$ may have different values in successive uses.  
\begin{lemma}\label{Le2}
\text{ }
\begin{enumerate}
\item For $\eps\ge 0$, $\sum_{k=1}^{\infty}k\hat\mu_{\eps}(Z_{k,\eps})\le \tilde{C}$.
\item $\lim_{\eps\to 0}|c_{\tau,\eps}-c_{\tau,p}|=0$.
\end{enumerate}
\end{lemma}
\begin{proof}
(1)  By Proposition \ref{prop1}, and the fact that the $L^{\infty}$-norm (w.r.t. $m$) is bounded by the BV-norm, we have
\begin{equation*}
\begin{split}
\sum_{k=1}^{\infty}k\hat\mu_{\eps}(Z_{k,\eps})&\le ||\hat h_{\eps}||_{\infty}\sum_{k=1}^{\infty}k|Z_{k,\eps}|\le \tilde{C}(\frac{B}{1-\beta})\sum_{k=1}^{\infty}\frac{1}{k^{1/\alpha}}
 \le \tilde{C}.
\end{split}
\end{equation*}
To prove (2), we first observe that the constants $c_{\tau,\eps}$ and $c_{\tau,p}$ are less or equal to $1$; then
\begin{equation*}
|c_{\tau,\eps}- c_{\tau,p}|=\left |\frac{1}{\sum_{k=1}^{\infty}k\hat\mu_{\eps}(Z_{k,\eps})}-\frac{1}{\sum_{k=1}^{\infty}k\hat\mu_p(Z_{k})}\right | \le \sum_{k=1}^{\infty} k|\hat\mu_{\eps}(Z_k,\eps)-\hat\mu_p(Z_k)|. 
\end{equation*}
By (1) the previous series is uniformly convergent in $\eps$. Therefore, it is enough to show that for any $k$, $|\hat\mu_{\eps}(Z_k,\eps)-\hat\mu_p(Z_k) |$ converges to zero as $\eps\to 0$.  We have $$|\hat\mu_{\eps}(Z_k,\eps)-\hat\mu_p(Z_k) |=\frac{1}{m(\Delta)}\left|\int_{Z_{k,\eps}}\hat{h}_{\eps}dx-\int_{Z_{k}}\hat{h}dx\right|\le$$
$$
\frac{1}{m(\Delta)}\left|\int_{Z_{k,\eps}\cap Z_k}\hat{h}_{\eps}dx+\int_{Z_{k,\eps}\backslash (Z_{k,\eps}\cap Z_k)}\hat{h}_{\eps}dx-\int_{Z_k\cap Z_{k,\eps}}\hat{h}dx-\int_{Z_k\backslash(Z_k\cap Z_{k,\eps})}\hat{h}dx\right|\le
$$
$$
\frac{1}{m(\Delta)}\left[\int_{Z_{k,\eps}\cap Z_k}|\hat{h}_{\eps}-\hat{h}|dx+ 2 ||\hat h_{\eps}||_{\infty} m(Z_{k,\eps}\Delta Z_k)\right]
$$
and the first term in the square bracket goes to zero 
 because $\lim_{\eps\to 0}||\hat h_{\eps}-\hat h||_1=0$.
\end{proof}
\begin{lemma}\label{Le3}
 For $\eps\ge 0$, $x\in W_{k,\eps}$ and $k$ large we have
\begin{enumerate}
\item
$$|DT^{(n-k)}(T_{i}^{-1}T_{1}^{-(n-k-1)}x)|\ge \left(\frac{n}{k+2}\right)^{\eta_k},$$
where $i=2,3$, $\eta_k=\frac{d(k+2)}{k+2+d}$ for some $d>1$. 
\item $$\sum_{n=k+1}^{\infty}\frac{1}{|DT^{(n-k)}(T_{i}^{-1}T_{1}^{-(n-k-1)}x)|}  \lesssim  k.$$
\end{enumerate}
\end{lemma}
\begin{remark}
Before proving Lemma \ref{Le3} we need two observations:
\begin{itemize}
\item The same proof holds for $T_{\eps}$ with all the constants involved uniformly bounded in $\eps$ for $\eps$ small. Moreover it will be clear in the proof of the theorem below that we can also take $x$ not in $W_k$ but in one of the two similar sets adjacent to it: the proof will not change.
\item It will be extremely important to have the constant $d$ strictly larger than $1$. Working with the map $T(x)=x+4^{\alpha}x^{1+\alpha}$, $x\in [0,1/4]$, such a constant will be $d= c^{\alpha} \  4^{\alpha}(1+\alpha)$, where the constant $c$ satisfies $b_k\ge c k^{-\frac{1}{\alpha}}$. This is done in the next sublemma.
\end{itemize}
\end{remark}
\begin{sublemma}
 Let $b_k=T_1^{-1}b_{k-1}$, with $b_0=1/4$. Then there exists $c$ independent of $k$ for which  $b_k\ge c k^{-\frac{1}{\alpha}}$, $k\ge 1$  and  $d:= c^{\alpha} \  4^{\alpha}(1+\alpha)>1$.
\end{sublemma}
\begin{proof}  
We proceed as in  Lemma 3.2 in \cite{LSV}, but proving the lower  bound. Let us choose $c=\frac{1}{4 (1+\alpha)^{\frac{1}{\alpha}}}+\delta$, where $\delta$ is a small positive constant whose value will be fixed later on. Note that with this value of $c$, the quantity $d>1$. We now prove the first assertion of the sublemma by induction. Suppose it is true for $k$; if it is not true for $k+1$ we should have
$$
b_k=b_{k+1}(1+4^{\alpha} b_{k+1}^{\alpha})\le c(k+1)^{-\frac{1}{\alpha}}(1+4^{\alpha}c^{\alpha}(k+1)^{-1})
$$
which implies that
$
k^{-\frac{1}{\alpha}}\le (k+1)^{-\frac{1}{\alpha}}(1+4^{\alpha}c^{\alpha}(k+1)^{-1})
$
or
$
\left(1+\frac1k\right)^{\frac{1}{\alpha}}-1\le \frac{4^{\alpha}c^{\alpha}}{k+1}.
$
But $\left(1+\frac1k\right)^{\frac{1}{\alpha}}-1\ge \frac{1}{\alpha}\frac{1}{k+1},$ which in conclusion gives us
$
c^{\alpha}\ge \frac{1}{\alpha \ 4^{\alpha}}.
$
With the given  choice $c=\frac{1}{4 (1+\alpha)^{\frac{1}{\alpha}}}+\delta,$ we see that for $\delta$ small enough the preceding lower bound  is false and so the induction is restored provided we prove the first step of it, namely
$
b_1\ge \frac{1}{4 (1+\alpha)^{\frac{1}{\alpha}}}+\delta.
$
Now $b_1+4^{\alpha}b_1^{1+\alpha}=1/4$; suppose $b_1$ will not verify the previous lower bound, then we should have
$$
\frac14 \le \frac{1}{4 (1+\alpha)^{\frac{1}{\alpha}}}+\delta+4^{\alpha}\left(\frac{1}{4 (1+\alpha)^{\frac{1}{\alpha}}}+\delta\right)^{1+\alpha}.
$$
It is easy to check that this can never be true.
\end{proof}
\begin{proof}(Of Lemma \ref{Le3})
As we anticipated above, we first need
(1). We have
\begin{equation*}
\begin{split}
|DT^{(n-k)}(T_i^{-1}T_1^{-(n-k-1)}x)|&=\Pi_{m=0}^{n-k-1}|DT(T^mT_{i}^{-1}T_{1}^{-(n-k-1)}x)\\
& \ge \Pi_{m=1}^{n-k-1}\inf_{y\in W_{k+m}}|DTy|\ge\Pi_{m=1}^{n-k-1}DT(b_{k+m+1}).
\end{split}
\end{equation*} 
The last estimate is true because the derivative of $T$ is increasing on $[0,a_0)$. In particular, since $DT_{1}(x)=1+(1+\alpha)4^{\alpha}x^{\alpha}$  and $b_k\ge c \frac{1}{k^{1/\alpha}}$, where $c$ is the constant given in the sublemma, we have
\begin{equation}\label{est1}
\begin{split}
|DT^{(n-k)}(T_{i}^{-1}T_{1}^{-(n-k-1)}x)|&\ge\Pi_{m=1}^{n-k-1}(1+\frac{d}{k+m+1})
=e^{\sum_{m=1}^{n-k-1}\log(1+\frac{d}{k+m+1})}.
\end{split}
\end{equation}
By the mean value theorem applied to the function $x\mapsto \log (1+x), x>0$ we immediately have 
\begin{equation}\label{est3}
\begin{split}
|DT^{(n-k)}(T_{i}^{-1}T_{1}^{-(n-k-1)}x)|&\ge e^{\frac{d}{1+\frac{d}{k+2}}\sum_{m=1}^{n-k-1}(\frac{1}{k+m+1})}
\ge e^{\frac{d}{1+\frac{d}{k+2}}\log\frac{n}{k+2}}=\left(\frac{n}{k+2}\right)^{\eta_k}.
\end{split}
\end{equation}
To prove (2) we sum over $n$ the estimate in (\ref{est3}) and we use the fact that $d>1$.
\end{proof}
\begin{proof} (Proof of Theorem \ref{main})
We have
\begin{equation*}
\begin{split}
||h_{\eps}-h_p||_1&\le\int_{\Delta}|h_{\eps}-h_p|dx+\sum_{k=1}^{\infty}\int_{W_{k}\setminus (W_{k,\eps}\cap W_k)}|h_{\eps}-h_p|dx\\
&+\sum_{k=1}^{\infty}\int_{W_{k,\eps}\cap W_k}|h_{\eps}-h_p|dx= {(I)}+ (II) +(III).
\end{split}
\end{equation*}
By Lemma \ref{Le1}
$$(I)=\int_{\Delta}|c_{\tau,\eps}\hat h_{\eps}-c_{\tau,p}\hat h_p|dx\le c_{\tau,p} \int_{\Delta}|\hat h_{\eps}-\hat h_p|dx +|c_{\tau,\eps}-c_{\tau,p}|\int_{\Delta}|\hat h_{\eps}|dx.$$
Therefore, by Proposition \ref{prop1} and Lemma \ref{Le2}, $(I)\to 0$ as $\eps\to 0$. To prove  that $(II)$ converges to zero we first obtain a bound on 
$\sup_{x\in W_{k}\setminus(W_{k,\eps}\cap W_k)}\left(|h_p(x)|+|h_{\eps}(x)|\right).$
Using (\ref{extension}), Proposition \ref{prop1} and Lemma \ref{Le3}, we have
\begin{equation*}
\begin{split}
\sup_{x\in W_{k}\setminus(W_{k,\eps}\cap W_k)}|h_p(x)|&\le\sup_{x\in W_{k}}c_{\tau,p}\sum_{n=k+2}^{\infty}\sum_{i=2}^3\frac{|\hat h_{p}(T_i^{-1}T_1^{-(n-k-2)}x)|}{|DT^{(n-k-1)}(T_i^{-1}T_1^{-(n-k-3)}x)|}\\
&\lesssim (\frac{B}{1-\beta})\ k.
\end{split}
\end{equation*}
A similar bound holds for $h_{\eps}$ by observing that the supremum should now be taken on an adjacent cylinder of $W_{k,\eps}$.
Consequently, since, as we already saw, $|b_k-b_{k-1}|\approx k^{-\frac{1}{\alpha}-1}, \ k\ge 1$, we obtain
\begin{equation*}
\begin{split}
(II)&\le 2(\frac{B}{1-\beta})\cdot\text{const}\sum_{k=1}^{\infty}k|b_{k}-b_{k-1}|\le \text{const}\sum_{k=1}^{\infty}k^{-\frac{1}{\alpha}}.
\end{split}
\end{equation*}
The uniform convergence of this series allows us to take the limit for $\eps \rightarrow 0$ inside and this will cancel the second contribution since $m(W_{k}\setminus (W_{k,\eps}\cap W_k))\rightarrow 0$ when $\eps\rightarrow 0$.
For the third one we have: 
\begin{equation*}
\begin{split}
&(III)=\sum_{k=1}^{\infty}\int_{W_{k,\eps}\cap W_k}|h_{\eps}(x)-h_p(x)|dx\\
&\le \sum_{k=1}^{\infty}\int_{W_{k,\eps}\cap W_k}| \sum_{n=k+1}^{\infty}\sum_{i=2}^3c_{\tau,p}\frac{\hat h_{p}(T_i^{-1}T_1^{-(n-k-1)}x)}{|DT^{(n-k)}(T_i^{-1}T_1^{-(n-k-1)}x)|}\\
&\hskip 4cm-c_{\tau,\eps}\frac{\hat h_{\eps}(T_{i,\eps}^{-1}T_{1,\eps}^{-(n-k-1)}x)}{|DT_{\eps}^{(n-k)}(T_{i,\eps}^{-1}T_{1,\eps}^{-(n-k-1)}x)|}|\\
&\le  \sum_{k=1}^{\infty}|c_{\tau,p}-c_{\tau,\eps}|\int_{W_{k,\eps}\cap W_k}\sum_{n=k+1}^{\infty}\sum_{i=2}^3\frac{|\hat h_{p}(T_i^{-1}T_1^{-(n-k-1)}x)|}{|DT^{(n-k)}(T_i^{-1}T_1^{-(n-k-1)}x)|}dx\\
&\hskip 1cm + \sum_{k=1}^{\infty}c_{\tau,\eps}\int_{W_{k,\eps}\cap W_k} \sum_{n=k+1}^{\infty}\sum_{i=2}^3|\frac{\hat h_{p}(T_i^{-1}T_1^{-(n-k-1)}x)}{|DT^{(n-k)}(T_i^{-1}T_1^{-(n-k-1)}x)|}\\
&\hskip 6cm -\frac{\hat h_{\eps}(T_{i,\eps}^{-1}T_{1,\eps}^{-(n-k-1)}x)}{|DT_{\eps}^{(n-k)}(T_{i,\eps}^{-1}T_{1,\eps}^{-(n-k-1)}x)|}|dx\\
&=A_1+A_2.
\end{split}
\end{equation*}
The quantity $A_1$ could be treated as the term (II) above: the integral inside the sum gives the summable contribution $k^{-\frac{1}{\alpha}}$ which will allow us to take afterwards the limit $|c_{\tau,p}-c_{\tau,\eps}|\rightarrow 0$ for $\eps\rightarrow 0$. The same argument shows that $A_2$ converges uniformly in $\eps$, but in order to take the limit inside the series, we have first of all to split $A_2$ into two supplementary terms:
\begin{equation*}
\begin{split}
A_2&\le\sum_{k=1}^{\infty}c_{\tau,\eps}\int_{W_{k,\eps}\cap W_k} \sum_{n=k+1}^{\infty}\sum_{i=2}^3|\frac{\hat h_{p}(T_i^{-1}T_1^{-(n-k-1)}x)}{|DT^{(n-k)}(T_i^{-1}T_1^{-(n-k-1)}x)|}\\
&\hskip 6cm -\frac{\hat h_{\eps}(T_{i,\eps}^{-1}T_{1,\eps}^{-(n-k-1)}x)}{|DT^{(n-k)}(T_i^{-1}T_1^{-(n-k-1)}x)|}|dx\\
&+\sum_{k=1}^{\infty}c_{\tau,\eps}\int_{W_{k,\eps}\cap W_k} \sum_{n=k+1}^{\infty}\sum_{i=2}^3|\frac{\hat h_{\eps}(T_{i,\eps}^{-1}T_{1,\eps}^{-(n-k-1)}x)}{|DT^{(n-k)}(T_i^{-1}T_1^{-(n-k-1)}x)|}\\
&\hskip 6cm -\frac{\hat h_{\eps}(T_{i,\eps}^{-1}T_{1,\eps}^{-(n-k-1)}x)}{|DT_{\eps}^{(n-k)}(T_{i,\eps}^{-1}T_{1,\eps}^{-(n-k-1)}x)|}|dx\\
&=A_2^*+A_2^{\dagger}.
\end{split}
\end{equation*}
To show that $A_2^*$ converges to zero as $\eps\to 0$ it will be sufficient to control the integral
$$\int_{W_{k,\eps}\cap W_k} \frac{1}{|DT^{(n-k)}(T_i^{-1}T_1^{-(n-k-1)}x)|}[\hat h_{p}(T_i^{-1}T_1^{-(n-k-1)}x)-\hat h_p(T_{i,\eps}^{-1}T_{1,\eps}^{-(n-k-1)}x))+
$$
$$
\hat h_p(T_{i,\eps}^{-1}T_{1,\eps}^{-(n-k-1)}x))-\hat h_{\eps}(T_{i,\eps}^{-1}T_{1,\eps}^{-(n-k-1)}x)]dx.
$$
We now make the change of variable $y_i=T_i^{-1}T_1^{-(n-k-1)}x\in Z_n$ and set $y'_i:=y'_i(y_i)=T_{i,\eps}^{-1}T_{1,\eps}^{-(n-k-1)}(T^{n-k}y_i)$. Then $y_i, y'_i \in Z_n\cup Z_{n\eps}$ and we rewrite the previous integral as
\begin{equation}\label{sum1}
\int_{Z_{n}}\left(|\hat h_p(y_i)-\hat h_p(y'_i)|+|\hat h_p(y'_i)-\hat h_{\eps}(y'_i)|\right)dy_i.
\end{equation}
We first have
$$\lim_{\eps\to 0}\int_{Z_{n}}|\hat h_p(y'_i)-\hat h_{\eps}(y'_i)|dy_i\le\lim_{\eps\to 0}||\hat h_p-\hat h_{\eps}||_1= 0.$$ 
We also have
 $\lim_{\eps\to 0}\int_{Z_{n}}|\hat h_p(y_i)-\hat h_{p}(y'_i)|dy_i=0,$
since by (1) of Lemma \ref{Le0} $\hat h_p$ is Lipschitz on $\Delta$ and $y'_i\to y_i$ as $\eps\to 0$.\\

To prove that $A_2^{\dagger}$ converges to $0$ as $\eps\to 0$,
it will be sufficient, after having factorized one of the inverse of the derivatives, to show that the ratio 
 $$\frac{DT_{\eps}^{(n-k)}(T_{i,\eps}^{-1}T_{1,\eps}^{-(n-k-1)}x)}{DT^{(n-k)}(T_i^{-1}T_1^{-(n-k-1)}x)}, \ x\in W_{k,\eps}\cap W_k
 $$
 goes to $1$. We begin to rewrite it as
 $$
 \Pi_{m=0}^{n-k-1}\frac{DT_{\eps}(T^my')}{DT(T^m y')}\frac{DT_{\eps}(T_{\eps}^my)}{DT_{\eps}(T^my')}
 $$
where we put $y:=T_{i,\eps}^{-1}T_{1,\eps}^{-(n-k-1)}x \in Z_{n,\eps}$ and $y':=T_i^{-1}T_1^{-(n-k-1)}x\in Z_n$ and we also recall that  $T_{\eps}^my\in W_{n-m,\eps}$ and  $T^my'\in W_{n-m}$. The first ratio $\frac{DT_{\eps}(T^my')}{DT(T^m y')}$ goes to one since for any $0\le x<b_0$: $\lim_{\eps\rightarrow 0}$ $DT_{\eps}(x)=DT(x)$. The second ratio can now be written in the form
$$
\left|\frac{DT_{\eps}(T_{\eps}^my)}{DT_{\eps}(T^my')}\right|=\exp\left[\left |\frac{D^2T_{\eps}}{DT_{\eps}}\right |_{\xi\in (T^m_{\eps}y,T^my')}\cdot |T^m_{\eps}y-T^my'|\right].
 $$
 Recall that the first and the second derivative are finite outside the origin; so we are left with proving that $|T^m_{\eps}y-T^my'|$ tends to $0$ when $\eps\rightarrow 0$. But $|T^m_{\eps}y-T^my'|= |T^m_{\eps}y-T^my|+|T^my-T^my'|$ and the first term goes to zero since $T^m_{\eps}$ converges uniformly to $T^m$ and the second term goes to zero by the continuity of $T^m$. This finishes the proof of part (1) of the theorem.
 \\
 To prove (2), we first use Proposition \ref{prop2} to obtain 
$$
\frac{\lambda_p}{1-\lambda_p}=\frac{\hat{\lambda}_p c_{\tau,r}}{(1-\hat{\lambda}_p)c_{\tau,l}}.
$$
Using (\ref{measure}) it follows immediately that $c_{\tau,r}=1$ and $c_{\tau,l}= \mu_l(\Delta_l)$, where $\Delta_l$ is the interval $(b_0,b)$. Therefore,
$$
\frac{\lambda_p}{1-\lambda_p}=\frac{\hat{\lambda}_p c_{\tau,r}}{(1-\hat{\lambda}_p)c_{\tau,l}}=\frac{1}{\mu_l(\Delta_l)}\lim_{\eps\rightarrow 0}\frac{\hat{\mu}_r(\hat{H}_{r,\eps})}{\hat{\mu}_l(\hat{H}_{l,\eps})}.$$
We now show that 
\begin{equation}\label{HF1}
\frac{1}{\mu_l(\Delta_l)}\lim_{\eps\rightarrow 0}\frac{\hat{\mu}_r(\hat{H}_{r,\eps})}{\hat{\mu}_l(\hat{H}_{l,\eps})}=\lim_{\eps\rightarrow 0}\frac{\mu_r(H_{r,\eps})}{\mu_l(H_{l,\eps})}
\end{equation}
which leads to the formula in part (2) of the theorem. We invoke formula (\ref{measure}) and the result which we obtained in part (1) of this theorem. We have
$H_{l,\eps}=I_l\cap T_{\eps}^{-1}I_r=(I_l\cap T_{\eps}^{-1}I_r)_l \cup (I_l\cap T_{\eps}^{-1}I_r)_r$, where $(I_l\cap T_{\eps}^{-1}I_r)_l=(I_l\cap T_{\eps}^{-1}I_r)\cap (0,b_0)$ and $(I_l \cap T_{\eps}^{-1}I_r)_r=(I_l \cap T_{\eps}^{-1}I_r)\cap(b_0,b)$.\\ 
Now, using (\ref{measure}) we obtain 
$$
\mu_{\eps}((I_l\cap T_{\eps}^{-1}I_r)_l)=c_{\tau,\eps} \sum_{n=2}^{\infty}\hat{\mu}_{\eps}(T_{\eps}^{-(n-1)}(I_l\cap T_{\eps}^{-1}I_r)_l\cap Z_{n,\eps})=
$$
$$
c_{\tau,\eps} \sum_{n=2}^{\infty}\hat{\mu}_{\eps}(\hat{T}^{-1}_{n,\eps}I_r\cap \Delta_l)=c_{\tau,\eps}\left[\hat{\mu}_{\eps}(\hat{T}^{-1}_{\eps}I_r\cap \Delta_l)-\hat{\mu}_{\eps}(\hat{T}^{-1}_{1,\eps}I_r\cap \Delta_l)\right]
$$
where we define
$ \hat{T}^{-1}_{n,\eps}:=\left({T_{\eps}^n}_{|_{Z_{n,\eps}}}  \right)^{-1}$. On the other hand
$$
\mu_{\eps}((I_l\cap T_{\eps}^{-1}I_r)_r)=c_{\tau,\eps}\hat{\mu}_{\eps}(\hat{T}^{-1}_{1,\eps}I_r\cap \Delta_l)
$$
since $(I_l\cap T_{\eps}^{-1}I_r)_r$ is inside the domain of induction. In conclusion we have proved that
$$
\mu_{\eps}(H_{l,\eps})=\mu_{\eps}(I_l\cap T_{\eps}^{-1}I_r)=c_{\tau,\eps} \hat{\mu}_{\eps}(\hat{H}_{l,\eps}).
$$
In a much easier way we immediately have
$$
\mu_{\eps}(H_{r,\eps})=\mu_{\eps}(I_r\cap T_{\eps}^{-1}I_l)=c_{\tau,\eps}\hat{\mu}_{\eps}(\hat{T}_{\eps}^{-1}\Delta_l\cap I_r)=c_{\tau,\eps}\hat{\mu}_{\eps}(\hat{H}_{r,\eps}).
$$
Therefore we have
\begin{equation}\label{HF2}
\frac{\mu_{\eps}(H_{r,\eps})}{\mu_{\eps}(H_{l,\eps})}=\frac{\hat{\mu}_{\eps}(\hat{H}_{r,\eps})}{\hat{\mu}_{\eps}(\hat{H}_{l,\eps})}.
\end{equation}
By (2) of Proposition \ref{prop1} and (1) of Theorem \ref{main}, we have:
\begin{itemize}
\item $\hat{\mu}_{\eps}(A)\rightarrow (1-\hat{\lambda}_p)\hat{\mu}_{r}(A)$, whenever $A$ is a mesurable set in $\Delta_r$.
\item $\mu_r(A)\leftarrow \frac{\mu_{\eps}(A)}{1-\lambda_p} $, whenever $A$ is a mesurable set in $I_r$.
\item $\hat{\mu}_{\eps}(A)\rightarrow \hat{\lambda}_p\hat{\mu}_{l}(A)$,  whenever $A$ is a mesurable set in $\Delta_l$.
\item $\mu_l(A)\leftarrow \frac{\mu_{\eps}(A)}{\lambda_p}$, whenever $A$ is a mesurable set in $I_l$.
\end{itemize}
Of course the same is true if $A$ depends on $\eps$ since, take for instance $A_{\eps}\subset I_r$,
$$
|\hat{\mu}_{\eps}(A_{\eps})- (1-\hat{\lambda}_p)\hat{\mu}_{r}(A_{\eps})|\le ||\hat{h}_{\eps}-\hat{h}_p||_1\rightarrow 0. 
$$
Putting together all that and using (\ref{HF2}) we get (\ref{HF1}):
\begin{equation*}
\begin{split}
\lim_{\eps\rightarrow 0}\frac{\mu_r(H_{r,\eps})}{\mu_l(H_{l,\eps})}&= \lim_{\eps\rightarrow 0}\frac{\lambda_p}{1-\lambda_p}\frac{\mu_\eps(H_{r,\eps})}{\mu_\eps(H_{l,\eps})}= \lim_{\eps\rightarrow 0}\frac{\lambda_p}{1-\lambda_p}\frac{\hat\mu_\eps(\hat H_{r,\eps})}{\hat\mu_\eps(\hat H_{l,\eps})}\\
&=\lim_{\eps\rightarrow 0}\frac{\hat\lambda_p}{(1-\hat\lambda_p)\mu_{l}(\Delta_l)}\frac{\hat\mu_\eps(\hat H_{r,\eps})}{\hat\mu_\eps(\hat H_{l,\eps})}=\lim_{\eps\rightarrow 0}\frac{1}{\mu_{l}(\Delta_l)}\frac{\hat\mu_r(\hat H_{r,\eps})}{\hat\mu_l(\hat H_{l,\eps})}.
\end{split}
\end{equation*}
 \end{proof} 
 {\bf Acknowledgement} We would like to thank anonymous referees for their comments. Their suggestions have improved the presentation of the paper.
 
\section{Appendix}
In the appendix we provide a method which can be used to determine the number of ergodic a.c.i.ms for maps similar to $T_{\eps}$. In particular we will show that for any $\eps>0$, the map $T_{\eps}$ defined in (\ref{map}) has exactly one a.c.i.m.
Let $\mathcal C_{(T_\eps)}:=\{I_i\}_{i=1}^6$ be the partition on which $T_{\eps}$ is piecewise monotonic. We introduce a directed graph associated with the perturbed map $T_\eps$, $\eps>0$, and we denote it by $G(T_{\eps})$\footnote{A similar graph can be found in \cite{BG} which is used to get an upper bound on the number on ergodic of a.c.i.ms when the modulus of the derivative of the map is greater than 2. Since in our case $\inf_{x}|T_{\eps}|=1$, we cannot use the results found in \cite{BG}.}.\\

\noindent $\bullet$ There is an arrow from $I_i\to I_j$ if and only if there exists a $k\ge1$ such that $T_{\eps}^k(I_i)\supseteq I_j$, $i,j\in\{1,\dots,6\}$.\\
$\bullet$ $I_j$ is said to be \textit{accessible} from $I_i$ if there exists aa \textit{arrow} in $G(T_{\eps})$ from $I_i$ to $I_j$.\\
$\bullet$ The \textit{accessible set} from $I_i$, denoted by $[I_i]$, consists of all intervals $I_j$ which are \textit{accessible} from $I_i$.\\

\begin{lemma}\label{new}
Let $\mu$ be a $T_{\eps}$ ergodic a.c.i.m\footnote{We know that there is at least one such measure since the corresponding induced map has an a.c.i.m.}. Then the support of $\mu$ contains $[I_i]$ for some $i=1,\dots,6$. 
\end{lemma}
\begin{proof}
We will first show that for any interval $J\subset I$, there exists an $n\ge 1$ such that $T_{\eps}^n(J)$ contains two partition points. Let $J\subset I_i$ for some $i$. Since $m(T_{\eps}(J))>m(J)$, there exists a $j\ge 1$ such that $T_{\eps}^j(J)$ contains a partition point in its interior. We consider all possible cases.
\begin{enumerate}
\item If $T_{\eps}^j(J)$ contains the partition point $0$, then obviously there exists a $k\ge 1$ such that $T_{\eps}^{j+k}(J)$ contains $[0,1/2]$.
\item The case of the partition point $3/8$ is the same as that of $0$. 
\item If $T_{\eps}^j(J)$ contains the partition point $1/4$ in its interior; i.e $T_{\eps}^j(J)\supset (p_1,p_2)$ with $1/4\in(p_1,p_2)$  . Then we observe that $T_{\eps}^k([1/4,p_2))\subseteq [1/4, 1]$ for all $k\ge 1$, and $\inf_{x\in[1/4,1]}|T_{\eps}'|>2$. Thus for some $k\ge 1$, $T_{\eps}^k([1/4,p_2))$ must contain two partition points (otherwise the length of iterates of the image will go to $\infty$ since the modulus of the derivative is bigger than 2). Thus, $T_{\eps}^{j+k}(J)$ contains two partition points.
\item The cases of the partition points $1/2, 5/8, 13/16, 1$ are similar to that of $1/4$.
\end{enumerate} 
Let $C$ denote the support of $\mu$. Since $C$ contains an interval $J$, $T^n(J)$, $n\ge 1$ contains two partition points, and $C$ is an invariant set, $C$ must contain (mod 0) an $I_i$. Consequently (by invariance) C must contain (mod 0) $[I_i]$.
\end{proof}
\begin{lemma}
For each $\eps>0$, $T_{\eps}$ has a unique ergodic a.c.i.m.
\end{lemma}
\begin{proof}
Observe that for each $i=1,\dots 6$,
$$[I_i]= \{I_1,I_2, I_3, I_4, I_5, I_6\}.$$
Thus by Lemma \ref{new}, and the fact that ergodic a.c.i.ms must have disjoint supports, $T_{\eps}$ has a unique a.c.i.m. 
\end{proof}  
\bibliographystyle{amsplain}

\end{document}